\documentclass[11pt, reqno]{amsart}

\usepackage{amssymb}
\usepackage{graphicx}
\usepackage{hyperref}
\usepackage{color, url}
\usepackage{enumitem}
\usepackage[all]{xy}

\newcommand{\Q}{\mathbb{Q}}
\newcommand{\Z}{\mathbb{Z}}

\newcommand{\OO}{\mathcal{O}}

\newcommand{\Ker}{\operatorname{Ker}}
\newcommand{\Coker}{\operatorname{Coker}}
\newcommand{\rank}{\operatorname{rank}}

\newcommand{\orb}{\mathrm{orb}}
\newcommand{\ab}{\mathrm{ab}}

\newcommand{\Int}{\operatorname{Int}}

\newtheorem{theorem}{Theorem}[section]
\newtheorem{proposition}[theorem]{Proposition}
\newtheorem{corollary}[theorem]{Corollary}
\newtheorem{lemma}[theorem]{Lemma}

\newtheorem{question}[theorem]{Question}
\theoremstyle{definition}
\newtheorem{definition}[theorem]{Definition}

\theoremstyle{remark}
\newtheorem{remark}[theorem]{Remark}

\begin{document}

\title[]{A ribbon knot which is not a symmetric union}

\author{Michel Boileau}
\address{Aix Marseille University, CNRS, UMR 7373 (I2M) \\
3 Pl. Victor Hugo, 13003 Marseille \\
France}
\email{michel.boileau@univ-amu.fr}

\author{Teruaki Kitano}
\address{Department of Information Systems Science, Faculty of Science and Engineering, Soka University \\
Tangi-cho 1-236, Hachioji, Tokyo 192-8577 \\
Japan}
\email{kitano@soka.ac.jp}

\author{Yuta Nozaki}
\address{
Department of Mathematics, Faculty of Science, Hokkaido University \\
Sapporo 060-0810 \\
Japan\vspace{-0.6em}}
\address{
International Institute for Sustainability with Knotted Chiral Meta Matter (WPI-SKCM$^2$), Hiroshima University \\
1-3-1 Kagamiyama, Higashi-Hiroshima, Hiroshima 739-8531 \\
Japan}
\email{nozaki@math.sci.hokudai.ac.jp}



\maketitle

\begin{abstract}
A basic open question motivated by the study of ribbon knot diagrams asks whether every ribbon knot can be presented as a symmetric union. 
In this article, we give a negative answer to this question by exhibiting a ribbon Montesinos knot which does not admit a symmetric union presentation.
\end{abstract}

\setcounter{tocdepth}{1}
\tableofcontents

\section{Introduction}
\label{sec:intro}
The symmetric union construction, introduced by Kinoshita and Terasaka in \cite{KiTe57} and extended by Lamm~\cite{Lam00}, generalizes the connected sum of a knot with its mirror image. 
This construction provides a geometric framework for generating ribbon knots.
Prime ribbon knots up to $10$ crossings and ribbon $2$-bridge knots admit symmetric union presentations (see \cite{Lam00, Lam21JKTR}).
This is true also for all but $15$ prime ribbon knots with $11$ and $12$ crossings \cite{Lam21EM}.
Whether its converse holds is a basic open question in the study of ribbon knots, as it would provide a diagrammatic characterization of being ribbon.
See \cite[Question~5.3]{Lam00}, \cite[Conjecture 5.2]{Lam21JKTR}, \cite[Conjecture]{BeFe24}, and most recently \cite[Problem~1.51]{K3}.

\begin{question}\label{ques:symunion}
Does every ribbon knot admit a symmetric union presentation?
\end{question}

This question is particularly challenging as there are currently no known obstructions to a ribbon knot admitting a symmetric union presentation.
Furthermore, there is no a priori upper bound on the number of twist regions required for a symmetric union presentation of a given ribbon knot. 
In this article, we give a negative answer to this question by exhibiting a ribbon Montesinos knot, with five rational tangles, which does not admit a symmetric union presentation.
Unlike the candidates previously considered which were satellite knots (see \cite[Problem~1.51]{K3}), this knot is hyperbolic. 
This example suggests that the class of ribbon knots should be subdivided into two distinct types according to whether a ribbon knot admits a symmetric union or not.

For a knot $K \subset S^3$, let $\Sigma_2(K)$ denote the $2$-fold branched cover of $S^3$ branched over $K$. 
For $\alpha \geq 3$ odd and coprime to $\beta \neq 0$, let $\frak{b}(\alpha, \beta)$ denote the $2$-bridge knot whose $2$-fold branched cover is the lens space $L(\alpha, \beta)$.

\begin{theorem}\label{thm:notsymmetricunion} 
The Montesinos knot $K = K(\frac{1}{3}, -\frac{1}{3}, \frac{1}{3}, -\frac{1}{3}, \frac{9}{2})$ is a ribbon knot which does not admit a symmetric union presentation.
\end{theorem}

\begin{remark}
It is true also for any other Montesinos knot obtained from $K$ by permuting the rational tangles.
More precisely, $K(\frac{1}{3}, \frac{1}{3}, -\frac{1}{3}, -\frac{1}{3}, \frac{9}{2})$, $K(\frac{1}{3}, -\frac{1}{3}, -\frac{1}{3}, \frac{1}{3}, \frac{9}{2})$, $K(-\frac{1}{3}, \frac{1}{3}, \frac{1}{3}, -\frac{1}{3}, \frac{9}{2})$, and their mirror images.
\end{remark}

\begin{remark} 
In her thesis, K\"{o}se~\cite[Theorem~E]{Kos22Thesis} used the $2$-fold branched cover to characterize Montesinos knots which admit symmetric union presentations with a single twist region. 
Her result implies that the Montesinos knot $K(\frac{1}{3}, -\frac{1}{3}, \frac{1}{3}, -\frac{1}{3}, \frac{9}{2})$ cannot admit such a symmetric union presentation. 
The proof crucially uses the fact that there is a single twist region.
\end{remark}

Throughout this paper, rational tangles are defined by \cite[Figure~12.10(b)]{BZH14}, and $K(\frac{\beta_1}{\alpha_1}, \dots, \frac{\beta_r}{\alpha_r})$ coincides with $\mathfrak{m}(0; \frac{\alpha_1}{\beta_1}, \dots, \frac{\alpha_r}{\beta_r})$ in \cite[Chapter~12.D]{BZH14} if $|\alpha_i| > |\beta_i|$.

The proof relies on a detailed analysis of the fundamental groups of the $2$-fold branched covers. 
Specifically, we exploit a necessary condition on the epimorphism from the fundamental group of the $2$-fold branched cover of a knot with a symmetric union presentation onto that of its partial knot. 

\begin{figure}[h]
 \centering \includegraphics[width=\textwidth]{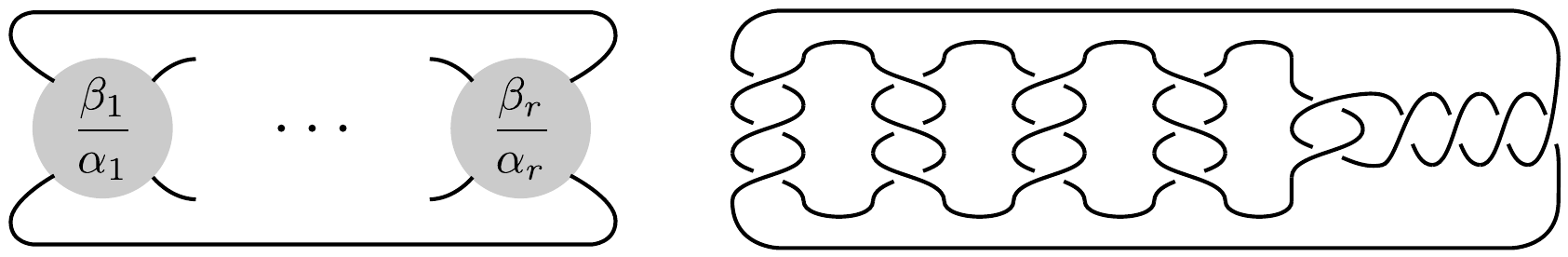}
 \caption{Montesinos knot $K(\frac{\beta_1}{\alpha_1}, \dots, \frac{\beta_r}{\alpha_r})$ and its example $K(\frac{1}{3}, -\frac{1}{3}, \frac{1}{3}, -\frac{1}{3}, \frac{9}{2})$.}
 \label{fig:Montesinos_ex}
\end{figure}

We have $\det(K) = 729$ for our example, so the following question remains open.

\begin{question}\label{ques:det1}
Does every ribbon knot $K$ with $\det(K) = 1$ admit a symmetric union presentation?
\end{question}

Note that there exist infinitely many ribbon knots $K$ with $\det(K) = 1$, for example, by \cite[Theorem~2.6]{Lam00}.

\subsection*{Acknowledgments}
This study was supported in part by JSPS KAKENHI Grant Numbers 23K12974 and 26K06793.

\section{Symmetric unions}
\label{sec:symmetric_union}

In this section, we recall the definition of the symmetric union construction.

\begin{definition}
\label{def:symmetric_union}
Let $D$ be an unoriented planar diagram of a knot $K_D$ and let $D^*$ be the diagram obtained from $D$ by reflecting $D$ across an axis $\Delta$ in the plane.
Let $B_0, B_1,\dots, B_k$ be balls along the axis $\Delta$, each of which is invariant by the reflection $\rho$ through $\Delta$ and intersects $D$ in a trivial arc.
One replaces the trivial tangle $(B_0, B_0 \cap (D \cup D^*))$ by a $\infty$-tangle to get the connected sum of the diagrams $D$ and $D^*$.
For $1\leq i \leq k$, one replaces each trivial tangle $(B_i, B_i \cap (D\sharp D^*))$ by a $n_i$-tangle, where $n_i \in \Z$.
The knot diagram $(D \cup D^*)(\infty, n_1, \dots, n_k)$ obtained from $D\cup D^*$ in this way is called a \emph{symmetric union} of the diagram $D$ and $D^*$.
A knot which admits a diagram $(D \cup D^*)(\infty, n_1, \dots, n_k)$ is said to admit a \emph{symmetric union presentation} with \emph{partial knot} $K_D$, where $K_D$ corresponds to the closure of the diagram $D$ such that $(D \cup D^*)(0, 0, \dots, 0) = K_D \cup K_{D^*}$.
See Figure~\ref{fig:symmetric_union}.
\end{definition}

\begin{figure}[h]
 \centering \includegraphics[width=\textwidth]{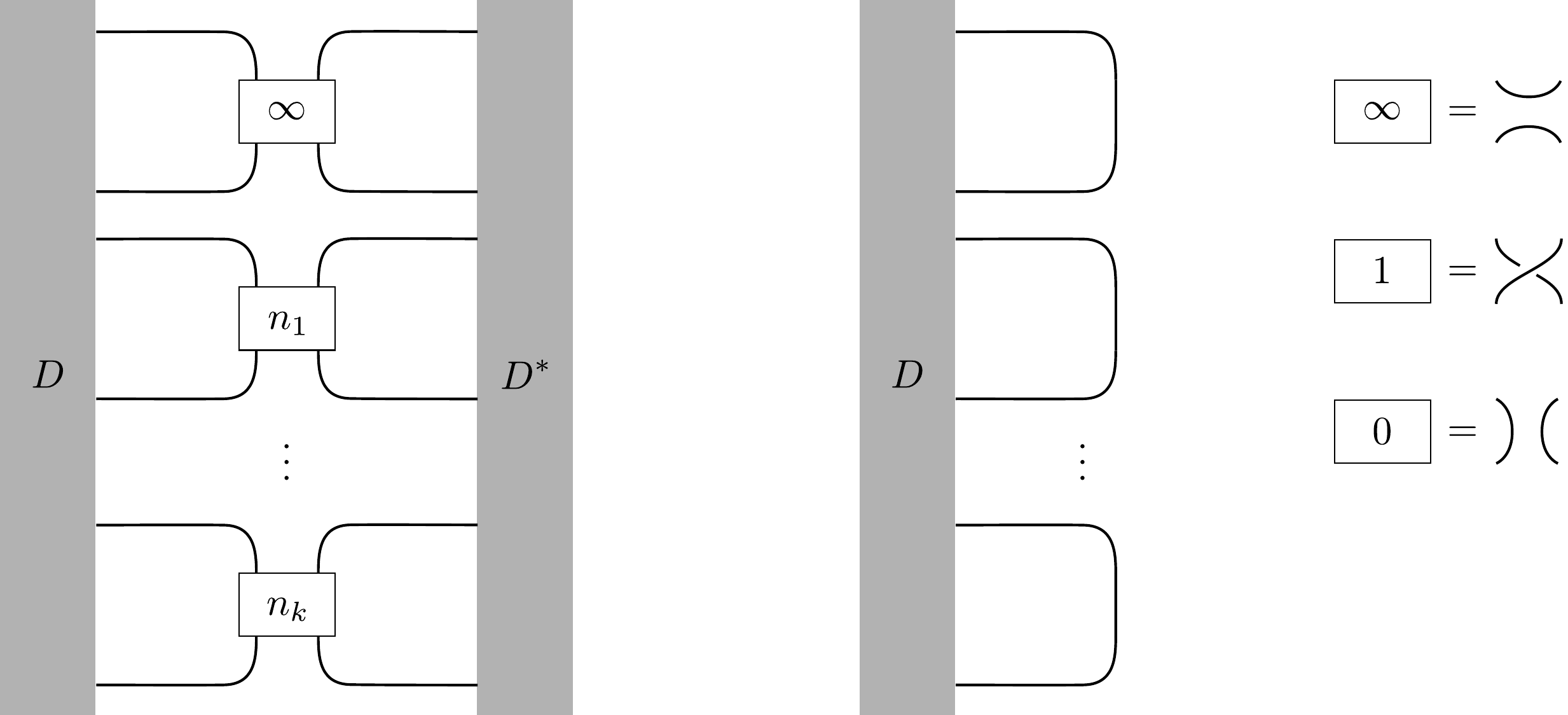}
 \caption{Symmetric union $(D \cup D^*)(\infty, n_1, \dots, n_k)$ and its partial knot.}
 \label{fig:symmetric_union}
\end{figure}

A key result for the study of symmetric unions is the existence of an epimorphism between the $\pi$-orbifold groups of a symmetric union and of its partial knot.
See \cite[Theorem~3.3]{Lam00} and also \cite{BKN25, BKN26}. 
This epimorphism lifts to an epimorphism between the fundamental groups of the $2$-fold branched covers of the symmetric union and of its partial knot. 
The next proposition gives a topological construction of this lift, which shows that it is induced by a map of degree zero between the $2$-fold branched covers. 
This last property is needed for the proof of Theorem~\ref{thm:notsymmetricunion}.

\begin{proposition}\label{prop:orbifold}
Let $K$ be a symmetric union with partial knot $K_D$. 
Then there is a $\pi_1$-surjective map $f \colon \Sigma_2(K) \to \Sigma_2(K_D)$ of degree zero.
\end{proposition}

\begin{proof}
First note that $K_{D}^* = K_{D^*}$.
Assume that the symmetric union diagram of $K$ lies on a great $2$-sphere $S \subset S^3$ and that the axis $\Delta$ of symmetry lies in the intersection of $S$ with an orthogonal great sphere $S' \subset S^3$. 
The $2$-sphere $S'$ splits $S^3$ into two balls $B$ and $B^*$.
Then, the reflection through the $2$-sphere $S'$ induces an orientation-reversing involution $\rho$ of the orbifold connected sum $\OO(K_D \sharp K_{D^*}) = \OO(K_D) \sharp_{S'(2,2)} \OO(K_{D^*})$ along the spherical $2$-orbifold $S'(2,2)$ with underlying space $S'$. 
This involution leaves invariant the $3$-balls $B_1, \dots, B_k$ and exchanges the two sides $\OO^{0}(K_D) = \OO(K_D) \cap B$ and $\OO^{0}(K_{D^*}) = \OO(K_{D^*}) \cap B^*$ of the connected sum $\OO(K_D) \sharp_{S'(2,2)} \OO(K_{D^*})$.
Let 
\[
X = \OO(K) \setminus \underset{i=1}{\overset{k}{\bigsqcup}} \Int B_i 
 = \OO(K_D) \sharp_{S^2(2,2)} \OO(K_{D^*})\setminus \underset{i=1}{\overset{k}{\bigsqcup}} \Int B_i. 
\]
Then, $\rho$ induces an orientation-reversing involution $\rho_0$ on the orbifold $X$ which exchanges the sides $X \cap B $ and $X \cap B^*$.

Let $q \colon \Sigma_2(K) \to \OO(K)$, $q_{D} \colon \Sigma_2(K_D) \to \OO(K_D)$, and $q_{D^*} \colon \Sigma_2(K_{D^*}) \to \OO(K_{D^*})$ be the natural projections. 
Then $\widetilde{X} =q^{-1}(X)$ is a submanifold of $\Sigma_2(K)$ such that $\Sigma_2(K) \setminus \Int \widetilde{X} = \underset{i=1}{\overset{k}{\bigsqcup}} q^{-1}(B_i)$ is a collection of solid tori. 
The manifold $\widetilde{X}$ is also a submanifold of the connected sum $\Sigma_2(K_D) \sharp \Sigma_2(K_{D^*})$ and the orientation-reversing involution $\rho$ lifts to an orientation-reversing involution $\tilde{\rho}$ on $\Sigma_2(K_D) \sharp \Sigma_2(K_{D^*})$ which preserves $\widetilde{X}$ and exchanges the sides $\Sigma_{2}^{0}(K_D) = q_{D}^{-1}(\OO^{0}(K_{D}))$ and $\Sigma_{2}^{0}(K_{D^*}) = q_{D^*}^{-1}(\OO^{0}(K_{D^*}))$. 
Hence, $\tilde{\rho}$ induces an orientation-reversing involution $\tilde{\rho}_0$ on the submanifold $\widetilde{X}$ which exchanges the sides $\Sigma_{2}^{0}(K_D) \cap \widetilde{X}$ and $\Sigma_{2}^{0}(K_{D^*}) \cap \widetilde{X}$.

Therefore, one can define a map $f \colon \widetilde{X} \to \Sigma_{2}^{0}(K_D) \cap \widetilde{X}$ by $f(x)=x$ if $x\in \Sigma_{2}^{0}(K_D) \cap \widetilde{X}$ and $f(x)=\tilde{\rho}_0(x)$ otherwise.
By construction, $f$ is of degree zero and induces an epimorphism $f_{\ast} \colon \pi_{1}(\widetilde{X}) \twoheadrightarrow \pi_{1} (\Sigma_{2}^{0}(K_D) \cap \widetilde{X})$. 
Moreover, $\pi_{1} (\Sigma_{2}^{0}(K_D)  \cap \widetilde{X}) = \pi_1(\Sigma_{2}^{0}(K_D) )= \pi_1(\Sigma_{2}(K_D))$ since $\Sigma_{2}^{0}(K_D)  \setminus \Int(\Sigma_{2}^{0}(K_D)  \cap \widetilde{X}) = \underset{i=1}{\overset{k}{\bigsqcup}} q_{D}^{-1}(B_i)$ is a collection of $3$-balls. 
The map $f$ sends each boundary torus $T_i = q^{-1}(\partial B_i) \subset \partial \widetilde{X}$ to the $2$-sphere $q_{D}^{-1}(\partial B_i)$ for $1 \leq i \leq k$. 
Therefore, $f$ can be extended by a map from the solid tori $\Sigma_2(K) \setminus \widetilde{X}$ to the balls $\Sigma^{0}_2(K_D) \setminus f(\widetilde{X})$. 
This gives a $\pi_1$-surjective map $f \colon \Sigma_2(K) \to \Sigma_2(K_D)$ of degree zero.
\end{proof}

Throughout this paper, the \emph{rank} $\rank(G)$ of a finitely generated group $G$ is defined to be the minimal number of elements needed to generate $G$.
The following proposition relates the ranks of the first homology groups of the $2$-fold branched covers of a symmetric union and of its partial knot.
Recall that a finite abelian group $A$ is isomorphic to a direct sum $\Z/d_{1}\Z \oplus \cdots \oplus \Z/d_{r}\Z$ of finite cyclic groups, where $d_1\geq 2$, $d_i$ divides $d_{i+1}$ for $1 \leq i \leq r-1$, and $r =\rank A$. 
The invariant factors $\{d_1, \ldots, d_r\}$ are unique. 

\begin{proposition}[{\cite[Proposition~7.10]{BKN26}}]
\label{prop:rank}
If $K$ admits a symmetric union presentation with partial knot $K_D$, then
\[
\rank H_1(\Sigma_2(K_D); \Z) \leq \rank H_1(\Sigma_2(K); \Z) \leq 2\rank H_1(\Sigma_2(K_D); \Z).
\]
\end{proposition}

\begin{remark}
By considering the Gordon-Litherland forms for knots, we can show that Proposition~\ref{prop:rank} is also true for a generalization of a symmetric union such that each twist can be knotted, that is, the boundary of any band in a $3$-ball.
Our argument in the rest of this paper actually implies that the ribbon knot $K(\frac{1}{3}, -\frac{1}{3}, \frac{1}{3}, -\frac{1}{3}, \frac{9}{2})$ is not even such a generalization of a symmetric union.
\end{remark}

\section{$K(\frac{1}{3}, -\frac{1}{3}, \frac{1}{3}, -\frac{1}{3}, \frac{9}{2})$ is a ribbon knot}
\label{sec:ribbon}

We briefly recall some notations about Montesinos knots and their Seifert fibered $2$-fold branched covers, see \cite[Chapter~12.D]{BZH14}. 
For $r \geq 1$, $\alpha_i \geq 2$, and $\beta_i \neq 0$ coprime to $\alpha_i$, let $K = K(\frac{\beta_1}{\alpha_1}, \dots, \frac{\beta_r}{\alpha_r})$ denote the Montesinos knot with $r$ rational tangles of slopes $\beta_i/\alpha_i \in \Q$.
The $2$-fold branched cover of $K$ is the closed orientable Seifert fibered $3$-manifold $\Sigma_2(K) = V(0; e_0; \frac{\beta_1}{\alpha_1}, \dots, \frac{\beta_r}{\alpha_r})$ with base $S^2(\alpha_1, \dots, \alpha_r)$ a $2$-dimensional orientable orbifold with underlying space $S^2$ and $r$ singular points with branching indices $\alpha_i$ corresponding to the $r$ exceptional fibers of types $(\alpha_i, \beta_i)$ (\cite{Mon73}).
See also \cite[Chapter~12.D]{BZH14}. 
Its rational Euler number $e_0 = \sum_{i=1}^{r} \frac{\beta_i}{\alpha_i} \in \Q$ satisfies that $|{\det(K)}| = |H_1(\Sigma_2(K); \Z)| = |e_0| \prod_{i=1}^{r} \alpha_i$ (see \cite[Corollary~6.2]{JaNe83}).

The goal of this section is to show the next result.

\begin{proposition}\label{prop:ribbon}
The Montesinos knot $K(\frac{1}{3}, -\frac{1}{3}, \frac{1}{3}, -\frac{1}{3}, \frac{9}{2})$ is a ribbon knot.
\end{proposition}

A knot $K \subset S^3$ is ribbon if and only if one can attach $n$ pairwise disjoint embedded bands which cut $K$ into a trivial link $L_0$ with $n+1$ unknotted components. 
Each band can be knotted and twisted, but it must meet the knot $K$ only along the attaching ends of the band. 
Then the knot $K$ bounds a disk with ribbon singularities obtained by band-sum of the $n+1$ pairwise disjoint disks bounding the components of $L_0$ and by making the bands transverse to the disks.

The following lemma follows from \cite[Section~1 c)]{OrRa68}.
See also \cite[Section~5.4(i)]{Orl72}, where the difference in sign depends on the orientation convention for Seifert fibered spaces and lens spaces. 
For the sake of completeness, we give a different proof.

\begin{lemma}\label{lem:montesinos2bridge}
Let $\alpha \geq 2$ and let $\beta > 0$ be coprime to $\alpha$.
Then the Montesinos knot $K(\frac{\beta}{\alpha})$ is the $2$-bridge knot $\frak{b}(\beta, \alpha)$.
\end{lemma}

\begin{proof}
Let $V = D^2 \times S^1$ be the exterior of an unknot $U$ in $S^3$ with the trivial circle foliation whose fiber on $\partial V$ is a meridian $m$ of $U$ and a section is the preferred longitude $\ell$ of $U$ which bounds a meridian disk in $V$.
The $2$-fold branched cover of the Montesinos knot $K(\frac{\beta}{\alpha})$ is the Seifert fibered manifold $V(0; \frac{\beta}{\alpha}; \frac{\beta}{\alpha})$ with base $S^2$ and a single exceptional fiber of type $\frac{\beta}{\alpha}$. 
It can be obtained by Dehn filling along the simple closed curve $\alpha \ell -\beta m$ on $\partial V$ (see \cite{JaNe83}). 
It means that $\Sigma_2(K(\frac{\beta}{\alpha}))$ can be obtained by Dehn surgery of $S^3$ on $U$ along the slope $-\frac{\beta}{\alpha}$, and hence it is homeomorphic to the lens space $L(\beta, \alpha)$. 
Since $2$-bridge knots are classified by their $2$-fold branched covers (see \cite[Satz~6]{Sch56}), the Montesinos knot $K(\frac{\beta}{\alpha})$ is the $2$-bridge knot $\frak{b}(\beta, \alpha)$. 
\end{proof}

\begin{lemma}\label{lem:ribbon}
Let $r \geq 1$.
If the Montesinos knot $K(\frac{\beta_1}{\alpha_1},\dots, \frac{\beta_r}{\alpha_r})$ is ribbon, then the Montesinos knot $K(\frac{1}{n}, -\frac{1}{n}, \frac{\beta_1}{\alpha_1}, \dots, \frac{\beta_r}{\alpha_r})$ is ribbon for $n \geq 1$. 
\end{lemma}

\begin{proof} 
Attaching a band between the two adjacent strands with crossing number $n$ and $-n$ as shown in Figure~\ref{fig:band_sum} yields the disjoint union of an unknot $U$ and of the Montesinos knot $K(\frac{\beta_1}{\alpha_1},\dots, \frac{\beta_r}{\alpha_r})$. 
Since this Montesinos knot is ribbon, one can further attach to $K(\frac{\beta_1}{\alpha_1},\dots, \frac{\beta_r}{\alpha_r})$ finitely many pairwise disjoint embedded bands which cut the knot into a disjoint union of unknots. 
This shows that the Montesinos knot $K(\frac{1}{n}, -\frac{1}{n}, \frac{\beta_1}{\alpha_1},\dots, \frac{\beta_r}{\alpha_r})$ is ribbon.
\end{proof}

\begin{figure}[h]
 \centering \includegraphics[width=\textwidth]{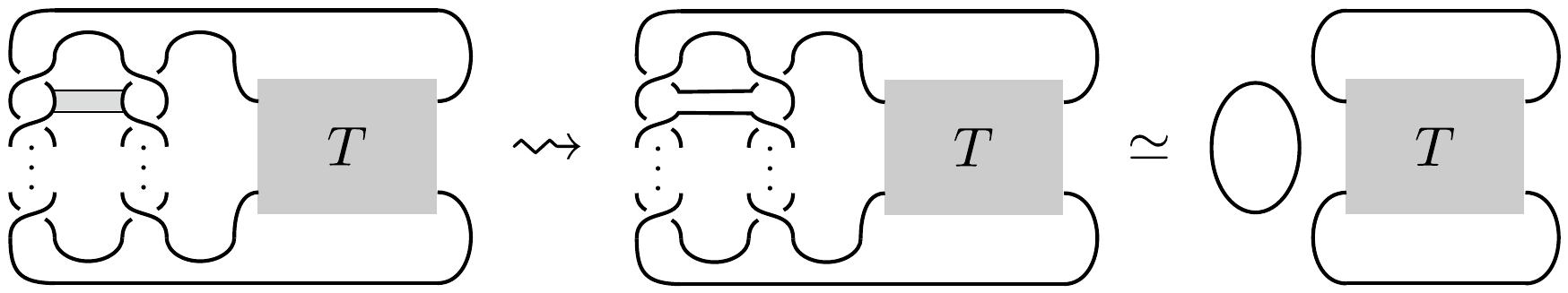}
 \caption{A band surgery from $K(\frac{1}{n}, -\frac{1}{n}, \frac{\beta_1}{\alpha_1},\dots, \frac{\beta_r}{\alpha_r})$ to $U\sqcup K(\frac{\beta_1}{\alpha_1},\dots, \frac{\beta_r}{\alpha_r})$, where $T$ represents $r$ rational tangles.}
 \label{fig:band_sum}
\end{figure}

Proposition~\ref{prop:ribbon} follows now from Lemmas~\ref{lem:montesinos2bridge} and \ref{lem:ribbon}. 
Here, Lemma~\ref{lem:montesinos2bridge} shows that the Montesinos knot $K(\frac{9}{2})$ is the $2$-bridge knot $\frak{b}(9, 2) = 6_1$ which is a symmetric union with partial knot $3_1$, and thus it is a ribbon knot.

\section{Proof of Theorem~\ref{thm:notsymmetricunion}}\label{sec:notsym}

We argue by contradiction and assume that the Montesinos knot $K =  K(\frac{1}{3}, -\frac{1}{3}, \frac{1}{3}, -\frac{1}{3}, \frac{9}{2})$ admits a symmetric union presentation with partial knot $K_D$. 
The proof of Theorem~\ref{thm:notsymmetricunion} splits into two cases according to whether $\Sigma_2(K_D)$ is aspherical or not. 
In the non-aspherical case, the structure of $\pi_1(\Sigma_2(K_D))$ is highly constrained by the orbifold theorem, which allows a direct combinatorial argument. 
The aspherical case requires a different approach based on the degree-zero map of Proposition~\ref{prop:orbifold} and results on Seifert $3$-manifolds. 
Finally, Propositions~\ref{prop:nonaspherical} and \ref{prop:aspherical} show that the Montesinos knot $K(\frac{1}{3}, -\frac{1}{3}, \frac{1}{3}, -\frac{1}{3}, \frac{9}{2})$ cannot admit a symmetric union presentation.

\begin{proposition}\label{prop:nonaspherical} 
The knot $K(\frac{1}{3}, -\frac{1}{3}, \frac{1}{3}, -\frac{1}{3}, \frac{9}{2})$ cannot admit a symmetric union presentation with a partial knot $K_D$ such that $\Sigma_2(K_D)$ is not aspherical.
\end{proposition}

\begin{proposition}\label{prop:aspherical} 
The knot $K(\frac{1}{3}, -\frac{1}{3}, \frac{1}{3}, -\frac{1}{3}, \frac{9}{2})$ cannot admit a symmetric union presentation with a partial knot $K_D$ such that $\Sigma_2(K_D)$ is aspherical.
\end{proposition}

\section{The non-aspherical case: proof of Proposition~\ref{prop:nonaspherical}}

Let $K = K(\frac{1}{3}, -\frac{1}{3}, \frac{1}{3}, -\frac{1}{3}, \frac{9}{2})$.
Then $\Sigma_2(K)$ is the Seifert fibered manifold $V(0; \frac{9}{2}; \frac{1}{3}, -\frac{1}{3}, \frac{1}{3}, -\frac{1}{3}, \frac{9}{2})$ with base the $2$-dimensional orbifold $S^2(2, 3, 3, 3, 3)$ whose underlying space is $S^2$ having four singular points with branching indices $3$ and one with branching index $2$. 
Assume that $K$ admits a symmetric union presentation with partial knot $K_D$. 
We determine the first homology groups of $\Sigma_2(K)$ and $\Sigma_2(K_D)$.

\begin{lemma}\label{lem:homology}
$H_1(\Sigma_2(K); \Z) \cong (\Z/3\Z)^2 \oplus \Z/9^2\Z$.
Therefore, $\rank H_1(\Sigma_2(K); \Z) = 3$.
\end{lemma}

\begin{proof}
$\Sigma_2(K)$ is the Seifert fibered manifold $V(0;\frac{9}{2}; \frac{1}{3}, -\frac{1}{3}, \frac{1}{3}, -\frac{1}{3}, \frac{9}{2})$. 
By \cite[Corollary~6.2]{JaNe83}, $H_1(\Sigma_2(K); \Z)$ is generated by $x, y, z, u, v, h$ subject to the relations
\[
3x+h = 3y-h = 3u+h = 3v-h = 2z+ 9h = x+y+z+u+v=0,
\]
where $x, y, u, v,$ correspond to the generators associated to the singular fibers of order $3$, $z$ to the singular fiber of order $2$, and $h$ to the regular fiber.
Thus, $H_1(\Sigma_2(K); \Z) = \Coker A $ with
\[
A =
\begin{pmatrix}
3 & 0 & 0 & 0 & 0 & 1 \\
0 & 3 & 0 & 0 & 0 & 1 \\
0 & 0 & 0 & 0 & 2 & 1 \\
0 & 0 & 3 & 0 & 0 & 1 \\
0 & 0 & 0 & 3 & 0 & 1 \\
1 & -1 & 1 & -1 & 9 & 0
\end{pmatrix}.
\]
Using elementary row and column operations, we have
\[
A \leadsto
\begin{pmatrix}
3 & 3& -3 & 3 & -27 & 1 \\
0 & 3 & 0 & 0 & 0 & 1 \\
0 & 0 & 0 & 0 & 2 & 1 \\
0 & 0 & 3 & 0 & 0 & 1 \\
0 & 0 & 0 & 3 & 0 & 1 \\
1 & 0 & 0 & 0 & 0 & 0
\end{pmatrix}
\leadsto
\begin{pmatrix}
3 & 3 & -3 & 0 & -27 & 0 \\
0 & 3 & 0 & -3 & 0 & 0 \\
0 & 0 & 0 & -3 & 2 & 0 \\
0 & 0 & 3 & -3 & 0 & 0 \\
0 & 0 & 0 & 3 & 0 & 1 \\
1 & 0 & 0 & 0 & 0 & 0
\end{pmatrix}.
\]
Hence, $\Coker A$ is isomorphic to the cokernel of the submatrix 
\[
A'=
\begin{pmatrix}
3 & -3 & 0 & -27 \\
3 & 0 & -3 & 0 \\
0 & 0 & -3 & 2 \\
0 & 3 & -3 & 0 
\end{pmatrix}.
\]
Now $A'$ is transformed  by row and column operations as follows:
\[
A'\leadsto
\begin{pmatrix}
0 & -3 & 3 & -27 \\
3 & 0 & -3 & 0 \\
0 & 0 & -3 & 2 \\
0 & 3 & -3 & 0 
\end{pmatrix}
\leadsto
\begin{pmatrix}
0 & -3 & 0 & -27 \\
3 & 0 & -3 & 0 \\
0 & 0 & -3 & 2 \\
0 & 3 & 0 & 0 
\end{pmatrix}
\leadsto
\begin{pmatrix}
0 & 0 & 0 & -27 \\
3 & 0 & 0 & 0 \\
0 & 0 & -3 & 2 \\
0 & 3 & 0 & 0 
\end{pmatrix}.
\]
Therefore, $\Coker A \cong (\Z/3\Z)^2 \oplus \Coker A''$, where $A''=
\begin{pmatrix}
0 & -27 \\
-3 & 2 \\
\end{pmatrix}$.
Because $A''$ can be transformed into 
$
\begin{pmatrix}
1 & 0 \\
0 & 81 \\
\end{pmatrix}$, 
we conclude that $\Coker A \cong (\Z/3\Z)^2 \oplus \Z/9^{2}\Z$.
\end{proof}

Proposition~\ref{prop:rank} and Lemma~\ref{lem:homology} imply the following consequence for the homology of the $2$-fold branched cover of the partial knot $K_D$.

\begin{corollary}\label{cor:homology} 
$H_1(\Sigma_2(K_D); \Z)$ is isomorphic to either $\Z/3\Z \oplus \Z/9\Z$ or $(\Z/3\Z)^3$.
\end{corollary}

\begin{proof}
It follows from Proposition~\ref{prop:rank} and Lemma~\ref{lem:homology} that
\[
2 \leq \rank H_1(\Sigma_2(K_D); \Z) \leq 3.
\]
By \cite[Theorem 2.4]{Lam00}, $\vert H_1(\Sigma_2(K_D); \Z) \vert = \sqrt{\vert H_1(\Sigma_2(K); \Z) \vert} = 27$.
Therefore, according to its rank, $H_1(\Sigma_2(K_D); \Z)$ is isomorphic to  $\Z/3\Z \oplus \Z/9\Z$ or to $\Z/3\Z \oplus \Z/3\Z \oplus \Z/3\Z$.
\end{proof}

Then Proposition~\ref{prop:nonaspherical} follows from the next lemma.

\begin{lemma}\label{lemm:nonaspherical} 
The $2$-fold branched cover $\Sigma_2(K_D)$ is aspherical.
\end{lemma}

\begin{proof}
Suppose, to the contrary, that $\Sigma_2(K_D)$ is not aspherical.
Since it is an orientable rational homology sphere, it is either (I) an irreducible $3$-manifold with finite fundamental group or (II) a non-trivial connected sum.
In the case (I), $\Sigma_2(K_D)$ is a Seifert $3$-manifold with finite fundamental group by the orbifold theorem (see \cite{BoPo01}). 
Hence, it is a lens space or a Seifert fibered $\Z/2\Z$-homology sphere with base $S^2(2, 3, 3)$ or $S^2(2, 3, 5)$ by \cite[Theorem~2(i), (iii), (v) in Section~6.2]{Orl72}. 
Then, $H_1(\Sigma_2(K_D); \Z)$ is cyclic, which contradicts Corollary~\ref{cor:homology}.

In the case (II), $\pi_1(\Sigma_2(K_D))$ is a non-trivial free product since $S^3$ is the only closed $3$-manifold whose fundamental group is trivial (see \cite{MoTi07} for example).
By Proposition~\ref{prop:orbifold}, there is an epimorphism $\varphi \colon \pi_1(\Sigma_2(K)) \twoheadrightarrow \pi_1(\Sigma_2(K_D))$.
Since a non-trivial free product is centerless, $\varphi$ factors through an epimorphism $\bar{\varphi} \colon \pi_{1}^\orb(S^2(2, 3, 3, 3, 3)) \twoheadrightarrow \pi_1(\Sigma_2(K_D))$. 
So $\pi_{1}^\orb(S^2(2, 3, 3, 3, 3))$ surjects onto each factor of the free product $\pi_1(\Sigma_2(K_D))$. 
Since $\pi_{1}^\orb(S^2(2, 3, 3, 3, 3))$ is generated by torsion elements, no factor can be torsion-free. 
Hence, each factor is the fundamental group of a prime rational homology sphere with torsion elements. 
By \cite[Corollary~9.9]{Hem04}, it must be finite, and the manifold is a lens space or Seifert fibered with base the orbifold $S^2(2, 3, 3)$ or $S^2(2, 3, 5)$.

Therefore, $\Sigma_2(K_D)$ is a connected sum of $n_1$ lens spaces and $n_2$ Seifert fibered manifolds over $S^2(2, 3, 3)$ or $S^2(2, 3, 5)$ with $n_1 + 2n_2 \leq 4$ since $\pi_{1}^\orb(S^2(2, 3, 3, 3, 3))$ is generated by four elements of order $3$ and there is an epimorphism $\bar{\varphi} \colon \pi_{1}^\orb(S^2(2, 3, 3, 3, 3)) \twoheadrightarrow \pi_1(\Sigma_2(K_D))$.
The map $\bar{\varphi}$ induces an epimorphism $\bar{\varphi}' \colon\pi_{1}^\orb(S^2(2, 3, 3, 3, 3))^{\ab} \twoheadrightarrow H_1(\Sigma_2(K_{D}); \Z)$.
It follows from $\pi_{1}^\orb(S^2(2, 3, 3, 3, 3))^{\ab} \cong (\Z/3\Z)^3$ and Corollary~\ref{cor:homology} that $H_1(\Sigma_2(K_D); \Z) \cong (\Z/3\Z)^3$. 
Hence, $\Sigma_2(K_D)$ must be the connected sum of at least three manifolds with finite fundamental group since these manifolds have cyclic first homology groups. 
Then, the only possibility for $\Sigma_2(K_D)$ is that $\Sigma_2(K_D) \cong L(3,1) \sharp L(3,1)\sharp M$, where $\pi_1(M)$ is finite and $H_1(M; \Z) \cong \Z/3\Z$. 
Therefore, we have an epimorphism $\pi_1(\Sigma_2(K_D)) \twoheadrightarrow \Z/3\Z \ast \Z/3\Z \ast \Z/3\Z$.
Composing it with $\bar{\varphi}$, one gets an epimorphism $\psi \colon \pi_{1}^\orb(S^2(2, 3, 3, 3, 3)) \twoheadrightarrow \Z/3\Z \ast \Z/3\Z \ast \Z/3\Z$.
The map $\psi$ factors through an epimorphism $\bar{\psi} \colon \pi_{1}^\orb(S^2(3, 3, 3, 3)) \twoheadrightarrow \Z/3\Z \ast \Z/3\Z \ast \Z/3\Z$ because the elements of order $2$ in $\pi_1^\orb(S^2(2, 3, 3, 3, 3))$ are killed by $\psi$.
Since $\pi_{1}^\orb(S^2(3, 3, 3, 3))$ is generated by three elements of order $3$, there is an epimorphism $\Z/3\Z \ast \Z/3\Z \ast \Z/3\Z \twoheadrightarrow \pi_1^\orb(S^2(3, 3, 3, 3))$.
Composing this epimorphisms with $\bar{\psi}$, one gets an epimorphism from $\Z/3\Z \ast \Z/3\Z \ast \Z/3\Z$ to itself, which is an isomorphism since $\Z/3\Z \ast \Z/3\Z \ast \Z/3\Z$ is hopfian.
This implies that $\pi_1^\orb(S^2(3, 3, 3, 3))$ is isomorphic to $\Z/3\Z \ast \Z/3\Z \ast \Z/3\Z$, which is impossible.
\end{proof}

\section{The aspherical case: proof of Proposition~\ref{prop:aspherical}}
\label{sec:aspherical}

In this section, we provide the proof of Proposition~\ref{prop:aspherical} by showing Proposition~\ref{prop:seifertdegreezero} that there is no epimorphism $f_\ast \colon \pi_1(M) \to \pi_1(N)$ with $\deg f = 0$, where $M$ is the Seifert manifold $V(0; \frac{9}{2}; \frac{1}{3}, -\frac{1}{3}, \frac{1}{3}, -\frac{1}{3}, \frac{9}{2})$ and $N$ is an aspherical manifold. 
Note here that the Seifert manifold $M$ is irreducible with infinite fundamental group, and thus it is aspherical.
To analyze such a map $f$, we use arguments analogous to the proof by Reid, Wang, and Zhou in \cite[Theorem~2.1]{RWZ02}. 

\begin{proposition}\label{prop:seifertdegreezero} 
An orientable Seifert manifold $M$ with orbifold base $B_M = S^2(2, 3, 3, 3, 3)$ does not admit any $\pi_1$-surjective map $f\colon M \to N$ of degree zero, onto an orientable closed aspherical $3$-manifold $N$.
\end{proposition} 

Throughout this section, let $f \colon M \to N$ be a $\pi_1$-surjective map of degree zero.
Assume that $M$ is a Seifert fibered rational homology sphere whose orbifold base $B_M$ is $S^2(2, 3, 3, 3, 3)$ and that $N$ is orientable, closed, and aspherical.

\begin{lemma}\label{lem:fiber} 
There is a Seifert fibration on $N$ whose orbifold base $B_N$ has underlying space $S^2$ and at least three singular points.
Moreover, the induced epimorphism $f_{\ast} \colon \pi_1(M) \to \pi_1(N)$ satisfies $f_{\ast}(h_M) = h_N^d$ with $d \neq 0$, where $h_M \in \pi_1(M)$ and $h_N \in \pi_1(N)$ correspond to regular fibers of the Seifert fibrations of $M$ and $N$, respectively.
\end{lemma}

\begin{proof}
Since the orbifold base $B_M$ of $M$ is orientable, $\pi_1(M)$ contains an infinite cyclic center $Z_M = \langle h_M \rangle$ generated by the regular fiber.
Moreover, the orbifold fundamental group $\pi_{1}^\orb (B_M) = \pi_1(M)/Z_M$ is generated by torsion elements since $\vert B_M \vert = S^2$.
Suppose, to the contrary, that $f_{\ast}(h_M) =1 \in \pi_1(N)$. 
Then the epimorphism $f_{\ast}$ factors through $\pi_{1}^\orb(B_M)$ which is generated by torsion elements. 
Since $N$ is aspherical, $\pi_1(N)$ is torsion-free, which is a contradiction.
Therefore, $f_{\ast}(h_M)$ is a non-trivial element in the center of $\pi_1(M)$.

It follows from \cite[Theorem~1.1]{CaJu94} or \cite[Corollary~2]{Gab92} that $N$ admits a Seifert fibration with an orbifold base $B_N$ and such that the infinite cyclic subgroup $Z_N = \langle h_N \rangle$ generated by the regular fiber is central. 
Since $N$ is an aspherical rational homology sphere, $Z_N$ is the center of $\pi_1(N)$ (see \cite[Theorem~12.10]{Hem04} or \cite[Proposition~II.4.7]{JaSh79}). 
Hence, $f_{\ast}(h_M)\neq 1$ belongs to  $Z_N = \langle h_N \rangle$, and $f_{\ast}(h_M) = h_N^d$ with $d \neq 0$. 
Since $f_{\ast} \colon \pi_1(M) \twoheadrightarrow \pi_1(N)$ sends $Z_M$ to $Z_N$, it induces an epimorphism $\bar{f}_{\ast} \colon \pi_{1}^\orb(B_M) \twoheadrightarrow \pi_{1}^\orb(B_N)$. 
The map $\bar{f}_{\ast}$ induces an epimorphism $\vert \bar{f}_{\ast} \vert \colon \pi_1(\vert B_M \vert) \twoheadrightarrow \pi_1(\vert B_N \vert)$. 
Therefore, $\pi_1(\vert B_N \vert)= \{1\}$ and $\vert B_N \vert = S^2$.
Moreover, the orbifold base $B_N$ has at least three singular points by the hypothesis that $N$ is aspherical. 
\end{proof}

The proof of the next lemma appears at the beginning of \cite[Step~(2)]{RWZ02}.
For the sake of completeness, we give its proof.

\begin{lemma}\label{lem:missingfiber}
For $3$-manifolds $M$ and $N$ in Proposition~\ref{prop:seifertdegreezero}, the $\pi_1$-surjective map $f \colon M \to N$ of degree zero can be homotoped to a fiber-preserving map which misses a fibered tubular neighborhood $V_0$ of a regular fiber of $N$.
\end{lemma}

\begin{proof}
Since $M$ and $N$ are orientable aspherical manifolds, there is a unique Seifert fibration on $M$ and on $N$ up to isotopy, with an orientable base $S^2$ and at least three singular points by \cite{Scott_1985} and \cite{Wal67}.
Moreover, these Seifert fibrations are induced by effective, fixed-point-free, $S^1$-actions on $M$ and $N$ (see \cite[Theorem~2.1]{JaNe83}). 
Then, the proof of \cite[Proposition~2.4]{Ron93} shows that the map $f \colon M \to N$ can be homotoped into an equivariant fiber-preserving map with respect to these $S^1$-actions since $f_{\ast}(h_M) = h_N^d$ with $d \neq 0$ and the Seifert fibration on $N$ is unique up to isotopy. 
The proof uses a round handle decomposition of $M$ and \cite[Lemmas~2.2 and 2.3]{Ron93}, but it does not need any assumption on the degree of $f$.

Hence, after the homotopy, one gets a $\pi_1$-surjective, fiber-preserving map $M\to N$ of degree zero, which is again denoted by $f$. 
Let $\Sigma_M$ and $\Sigma_N$ be the sets of singular fibers of $M$ and $N$. 
One can further deform $f$ such that $f^{-1}(\Sigma_N)$ is a finite collection of fibers of $M$.
Let $W_0$ be a regular neighborhood of $\Sigma_N \cup f(\Sigma_M)$ in $N$. 
Then $f$ induces a proper, fiber-preserving map of degree zero:
\[
f_{\vert} \colon M^{0} = M \setminus \Int\big(f^{-1}(W_0)\big) \to 
N^0 = N \setminus \Int(W_0).
\]

Since $f$ is of degree zero, the restriction $f_{\vert} \colon M^0 \to N^0$ is of degree zero. 
The submanifolds $M^{0}$ and $N^0$ are trivial $S^1$-bundles over the punctured surfaces $B_{M}^{0}$ and $B_{N}^{0}$.
Since $f_{\ast}(h_M) = h_N^d$ with $d \neq 0$, $f_{\vert}$ induces a map $\bar{f}_{\vert} \colon B_{M}^{0} \to B_{N}^{0}$ of degree zero.
By Kneser's theorem~\cite{Kne30} (see also \cite{Sko87}), $\bar{f}_{\vert}$ can be properly homotoped to a map which misses a disk $D_0 \subset B_{N}^{0}$. 
This homotopy lifts to a fiber-preserving homotopy between $f_{\vert} \colon M^0 \to N^0$ and a fiber-preserving map $f'_{\vert} \colon M^0 \to N^0$ which misses a fibered tubular neighborhood $V_0$ of a regular fiber of $N^0$. 
Since $f_{\vert}$ and $f'_{\vert}$ are homotopic on $\partial M^0$, $f'_{\vert}$ extends to a fiber-preserving map $f' \colon M\to N$ which is homotopic to $f$ and misses $V_0$.
\end{proof}

\begin{proof}[Proof of Proposition~\ref{prop:seifertdegreezero}]
By Lemmas~\ref{lem:fiber} and \ref{lem:missingfiber}, there is a $\pi_1$-surjective, fiber-preserving map $f \colon M \to N$ such that $f(M) \subset N'$, where $N' = N \setminus \Int(V_0)$ and $V_0 \subset N$ is a fibered tubular neighborhood of a regular fiber of $N$.
Let $B_{N'} = B_N \setminus \Int(D_0)$ be the base of $N'$ with $D_0 \subset B_N$ being a disk and let $\pi' \colon N' \to B_{N'}$ be the quotient map onto the base.

Let $\overline{\mathcal{A}}$ be a finite collection of essential, disjoint, non-parallel and properly embedded arcs in $B_{N'}$ such that $A \cap \pi'(f(\Sigma_M) \cup \Sigma_N) = \emptyset$ for any $A\in \overline{\mathcal{A}}$ and $\overline{\mathcal{A}}$ cuts $B_{N'}$ into discal $2$-orbifolds (i.e., disks with at most one singular point). 
Such a collection $\overline{\mathcal{A}}$ is not empty because the orbifold $B_{N'}$ has at least three singular points, like $B_N$.
Then the preimages ${\mathcal{A}} = \pi'^{-1}(\overline{\mathcal{A}})$ give a finite collection of disjoint, essential, non-parallel, saturated, properly embedded annuli in $N'$ which miss $f(\Sigma_M) \cup \Sigma_N$, and ${\mathcal{A}}$ cuts $N'$ into Seifert fibered solid tori. 

One can deform $f$ via a homotopy so that each surface in $\mathcal{T} = f^{-1}(\mathcal{A})$ is incompressible in $M$ (see \cite[Lemma~6.5]{Hem04}).
The collection $\mathcal{T}$ of closed orientable surfaces is not empty.
Otherwise, $f(M)$ is contained in a Seifert fibered solid torus and $f_{\ast}(\pi_1(M))$ would be cyclic, which is impossible since $f_{\ast}(\pi_1(M)) = \pi_1(N)$.
No component of $\mathcal{T}$ can be isotopic to a surface $F$ transverse to the Seifert fibration of $M$ (i.e., horizontal). 
Otherwise, $F$ must be separating since $M$ is a rational homology sphere. 
Then, the restriction of $\pi \colon M \to B_M$ to $F$ induces an orbifold covering, and hence the base $B_M$ is non-orientable.
This is a contradiction. 
(One could also argue that the rational Euler number $e_0$ of $M$ is non-zero, while it is zero when there is a horizontal surface which gives a section of the map $\pi \colon M \to B_M$. 
One can also apply the proof of Case~1 of \cite[Lemma~3.5]{Ron93}. 
The argument here works even if $M$ is not a rational homology sphere.)

Hence, there is an isotopy of $f$ so that $\mathcal{T}$ is a non-empty finite collection of incompressible vertical tori in $M$ by \cite{Wal67}.
Let $T \in \mathcal{T}$,
then $T$ splits $M$ into two Seifert fibered submanifolds $S_1$ and $S_2$ with incompressible boundary $\partial S_1= \partial S_2 = T$. 
Since $f(T) \subset A \in \mathcal A$ and $f_{\ast}(h_M) = h_N^d$ with $d \neq 0$, the restriction $f_{\vert}\colon T \to A$ is homotopic to a fiber-preserving map which sends the base $S^1$ of the Seifert fibration on $T$ to the base $[0, 1]$ of the Seifert fibration on $A$. 
Hence, $\Ker({f_{\vert}}_{\ast})$ and the Seifert fiber on $T$ generate $\pi_1(T)$. 
It follows that $\Ker({f_{\vert}}_{\ast}) \cong \mathbb{Z}$ is generated by a simple closed curve on $T$ dual to the Seifert fiber, i.e., by a section $s$ of the Seifert fibration on $T$. 
Therefore, for each Seifert fibered submanifold $S_i$ ($i=1, 2$) with boundary $T$, the induced homomorphism $f_{\ast} \colon \pi_1(S_i) \to \pi_1(N)$ factors through $\pi_1(S_i(s))$, where $S_i(s)$ is the $3$-manifold obtained by Dehn filling the boundary $\partial S_i = T$ along the section $s$. 
Here, $S_i(s)$ is a closed Seifert fibered manifold such that the core of the glued solid torus is a regular fiber of $S_i(s)$, and hence $S_i(s)$ has the same number of singular fibers as $S_i$. 
Moreover, $f_{\ast}(\pi_1(S_i))$ must be infinite since $f_{\ast}(h_M) = h_N^d$ is of infinite order in $\pi_1(N)$ and belongs to $f_{\ast}(\pi_1(S_i))$. 
Since $f_{\ast} \colon \pi_1(S_i) \to \pi_1(N)$ factors through $\pi_1(S_i(s))$, $\pi_1(S_i(s))$ must be an infinite group for $i=1, 2$.

Recall that the orbifold base $B_M$ is $S^2(2, 3, 3, 3, 3)$. 
Note that the vertical torus $T$ projects to a simple closed curve $\gamma$ on $B_M$, which splits $B_M$ into two disks. 
Each disk contains at least two of the five singular points, otherwise $T$ bounds a fibered solid torus and is not incompressible. 
Hence, the only possible distribution of the singular points $\{2, 3, 3, 3, 3\}$ into the two disks is $(2, 3)$ and the bases of the Seifert fibered $3$-manifolds $S_1(s)$ and $S_2(s)$ must be $S^2(m, n)$ and $S^2(p, q, r)$ with $\{m, n, p, q, r\} = \{2, 3, 3, 3, 3\}$.
A closed, orientable, Seifert $3$-manifold with base $S^2(2, 3)$ has a finite fundamental group because it is a lens space or $S^3$.
Therefore, $\{m, n\} = \{3, 3\}$. 
It follows that $\{p, q, r\} = \{2, 3, 3\}$, which is not possible because a closed orientable Seifert fibered $3$-manifold with base $S^2(2, 3, 3)$ has a finite fundamental group, see \cite[Theorem~2(iii), Section~6.2]{Orl72}. 
This shows that the orbifold base $B_M$ cannot be $S^2(2, 3, 3, 3, 3)$ and finishes the proof of Proposition~\ref{prop:seifertdegreezero}.
\end{proof}


\end{document}